\documentclass[11pt,reqno]{amsart}
\usepackage{latexsym,amsmath}
\usepackage[utf8]{inputenc}
\usepackage{amsthm}
\usepackage{amssymb}
\usepackage{epsfig}
\usepackage{microtype} 
\usepackage{fullpage}
\usepackage{setspace}
\usepackage{hyperref}
\usepackage{url}
\usepackage[shortlabels]{enumitem} 

\DeclareMathOperator{\exx}{\textsl{ex}}
\DeclareMathOperator{\dcut}{\textsl{d}_{\square}}
\DeclareMathOperator{\Forb}{\textsl{Forb}}
\def\Forbhom^*{\Forb_{\textsl{hom}}^*} \def\ex{\exx^*}
\DeclareMathOperator{\TOWER}{{\tt TOWER}}
\DeclareMathOperator{\poly}{{\tt poly}}

\DeclareMathOperator{\dist}{\textsl{d}_1}
 \DeclareMathOperator{\homto}{\to}
\let\hom\relax \DeclareMathOperator{\hom}{\textsl{hom}}

\def\hat{\widehat}
\def\bar#1{\overline{#1}}

\def\eps{\varepsilon} \def\phi{\varphi} \def\cP{\mathcal{P}}
\def\cV{\mathcal{V}} \def\cU{\mathcal{U}} 
  \def\cR{\mathcal{R}}
\def\cG{\mathcal{G}} \def\cH{\mathcal{H}} \def\cF{\mathcal{F}}
\def\RR{\mathbb{R}} \def\NN{\mathbb{N}}  
\def\PP{\mathbb{P}} \def\({\left(} \def\){\right)} \def\>{\rangle}
\def\z{z_{\mathcal{F}}}
\def\<{\langle}
\let\epsilon\varepsilon

\newcommand{\quot}[2]{\mathchoice%
{\left.\raisebox{.1em}{$\displaystyle{#1}$}\kern-1pt/\raisebox{-.2em}{$\displaystyle{#2}$}\right.}
{\left.\raisebox{.1em}{${#1}$}\kern-1pt/\raisebox{-.2em}{${#2}$}\right.}
{\left.\raisebox{.1em}{$\scriptstyle{#1}$}\kern-1pt/\raisebox{-.2em}{$\scriptstyle{#2}$}\right.}
{\left.\raisebox{.1em}{$\scriptscriptstyle{#1}$}\kern-1pt/\raisebox{-.2em}{$\scriptscriptstyle{#2}$}\right.}
} \newcommand{\etal}{\textsl{et al.}}

\usepackage{datetime}
\shortdate
\yyyymmdddate
\settimeformat{ampmtime}
\date{\today, \currenttime}

\title{Estimating parameters associated with monotone properties}

\author[Hoppen]{Carlos Hoppen}
\address{Instituto de Matem\'atica, UFRGS, Avenida Bento
  Gon\c{c}alves, 9500, 91501-970 Porto Alegre, RS, Brazil {\rm(C. Hoppen)}} 
\email{choppen@ufrgs.br}

\author[Kohayakawa]{Yoshiharu Kohayakawa}
\address{Instituto de Matem\'atica e Estat\'istica, Universidade de
  S\~{a}o Paulo, Rua do Mat\~{a}o 1010, 05508-090 S\~{a}o Paulo,
  Brazil {\rm(Y. Kohayakawa and H. Stagni)}} 
\email{\{yoshi|stagni\}@ime.usp.br}

\author[Lang]{Richard Lang}
\address{Departamento de Ingenier\'ia Matem\'atica, Universidad de Chile, Beauchef 851, 
Torre Norte, Piso 4, Oficina 415, Santiago, Chile {\rm(R. Lang)}}
\email{rlang@dim.uchile.cl}

\author[Lefmann]{Hanno Lefmann} 
\address{Fakult\"at f\"ur Informatik, Technische
Universit\"at Chemnitz, Stra\ss{}e der Nationen 62, 09111 Chemnitz,
Germany {\rm(H. Lefmann)}}
\email{lefmann@informatik.tu-chemnitz.de}

\author[Stagni]{Henrique Stagni}

\thanks{C. Hoppen acknowledges the support of
    FAPERGS~(Proc.\,2233-2551/14-9) and CNPq (Proc.~448754/2014-2
    and~308539/2015-0). C. Hoppen and H. Lefmann acknowledge the
    support of CAPES and DAAD via PROBRAL (CAPES Proc.~408/13 and
    DAAD~57141126 and~57245206).  C. Hoppen,
    Y. Kohayakawa and H. Stagni thank FAPESP (Proc.~2013/03447-6) and
    NUMEC/USP (Project MaCLinC/USP) for their support. Y.  Kohayakawa
    was partially supported by FAPESP (2013/07699-0) and CNPq
    (310974/2013-5 and~459335/2014-6). H. Stagni was supported by
    FAPESP (2015/15986-4) and CNPq (141970/2015-4 and~459335/2014-6).
    R. Lang acknowledges support by Millennium Nucleus Information and Coordination in Networks ICM/FIC RC130003.
     Some results in this paper appeared in preliminary form in the
     Proceedings of APPROX/RANDOM 2016~\cite{hoppen_et_al:LIPIcs:2016:6658}.}

\begin{document} 
\onehalfspace
\allowdisplaybreaks[2]
\footskip=28pt
\newtheorem{theorem}{Theorem}[section]
\newtheorem{cor}[theorem]{Corollary}
\newtheorem{lemma}[theorem]{Lemma}
\newtheorem{fact}[theorem]{Fact}
\newtheorem{property}[theorem]{Property}
\newtheorem{corollary}[theorem]{Corollary}
\newtheorem{proposition}[theorem]{Proposition}
\newtheorem{claim}[theorem]{Claim}
\newtheorem{conjecture}[theorem]{Conjecture}
\newtheorem{definition}[theorem]{Definition}
\theoremstyle{definition}
\newtheorem{example}[theorem]{Example}
\newtheorem{remark}[theorem]{Remark}

\begin{abstract} There has been substantial interest in estimating the value of
a graph parameter, i.e., of a real-valued function defined on the set of finite
graphs, by querying a randomly sampled substructure whose size is independent of
the size of the input. Graph parameters that may be successfully estimated in
this way are said to be \emph{testable} or \emph{estimable}, and the
\emph{sample complexity} $q_z=q_z(\eps)$ of an estimable parameter $z$ is the
size of a random sample of a graph $G$ required to ensure that the value of
$z(G)$ may be estimated within an error of $\eps$ with probability at least
2/3. In this paper, for any fixed monotone graph property $\cP=\Forb(\cF)$,
we study the sample complexity of estimating a bounded graph parameter
$z_{\cF}$ that, for an input graph $G$, counts the number of 
\emph{spanning}
subgraphs of $G$ that satisfy $\cP$. 
To improve upon previous upper bounds on the sample
complexity, we show that the vertex set of any graph that satisfies a
monotone property $\cP$ may be partitioned equitably into a constant number of
classes in such a way that the cluster graph induced by the partition is not
far from satisfying a natural weighted graph generalization of $\cP$.
Properties for which this holds are said to be \emph{recoverable}, and the
study of recoverable properties may be of independent interest.
\end{abstract}
\maketitle

\section{Introduction and main results}

In the last two decades, a lot of effort has been put into finding constant-time
randomized algorithms (conditional on sampling) to gauge whether a combinatorial
structure satisfies some property, or to estimate the value of some numerical
function associated with this  structure. In this paper, we focus
on the graph case and, as usual, we consider algorithms that have the ability to
query whether any desired pair of vertices in the input graph is adjacent or
not. Let $\cG$ be the set of finite simple graphs and let $\cG(V)$ be the set of
such graphs with vertex set $V$. We shall consider subsets $\cP$ of~$\cG$ that
are closed under isomorphism, which we call \emph{graph properties}. To
avoid technicalities, we restrict ourselves to graph properties $\cP$ such that $\cP \cap
\cG(V) \neq \emptyset$ whenever $V \neq \emptyset$. For instance, this includes all nontrivial
\emph{monotone} and \emph{hereditary} graph properties, which are graph
properties that are inherited by subgraphs and by induced subgraphs,
respectively. Here, we will focus on monotone properties. The prototypical
example of a monotone property is $\Forb(F)$, the class of all graphs that do
not contain  a copy of a fixed graph $F$.
More generally, if $\cP$ is a
monotone property and $\cF$ contains all minimal graphs that are not in $\cP$,
then the graphs that lie in $\cP$ are precisely those that  do not contain a copy of an element of $\cF$.
This class of graphs will be denoted by $\cP=\Forb(\cF)$. The elements of
$\Forb(\cF)$ are said to be \emph{$\cF$-free}.

 A graph property $\mathcal{P}$ is said to be \emph{testable} if, for every
 $\eps>0$, there exist a positive integer $q_\cP=q_\cP(\eps)$, called the
 \emph{query complexity}, and a randomized algorithm $\mathcal{T}_\mathcal{P}$,
 called a \emph{tester}, which may perform at most $q_\cP$ queries in the input
 graph, satisfying the following property. For an $n$-vertex input graph
 $\Gamma$, the algorithm $\mathcal{T}_\mathcal{P}$ distinguishes with
 probability at least $2/3$ between the cases in which $\Gamma$ satisfies
 $\mathcal{P}$ and in which $\Gamma$ is $\eps$-far from satisfying
 $\mathcal{P}$, that is, in which no graph obtained from $\Gamma$ by the
 addition or removal of at most $\eps n^2/2$ edges satisfies $\mathcal{P}$.  This
 may be stated in terms of graph distances: given two graphs $\Gamma$ and
 $\Gamma'$ on the same vertex set $V(\Gamma) = V(\Gamma')$, we may define the \emph{normalized edit distance} between
 $\Gamma$ and $\Gamma'$ by $\dist(\Gamma,\Gamma')=\frac{2}{|V|^2}
 \left|E(\Gamma) \triangle E(\Gamma')\right|$, where $E(\Gamma) \triangle
 E(\Gamma')$ denotes the symmetric difference of their edge sets. If $\cP$ is a graph property, we
 let the distance between a graph $\Gamma$ and $\cP$ be
 \begin{equation*}
 \dist(\Gamma,\cP)=\min\{\dist(\Gamma,\Gamma')\colon
 V(\Gamma')=V(\Gamma) \textrm{ and }\Gamma'\in \cP\}.  
 \end{equation*} 
For instance, if $\Gamma=K_n$ and $\cP=\Forb(K_3)$, Tur\'{a}n's Theorem ensures that $\binom{n}{2}-\lfloor n^2/4 \rfloor$ 
edges need to be removed to produce a graph that is $K_3$-free. In particular, $\dist(K_n,\Forb(K_3)) \rightarrow 1/2$. 
Thus a graph property $\cP$ is testable if
there is a tester with bounded query complexity
 that distinguishes with probability at least $2/3$ between the cases
 $\dist(\Gamma,\cP)=0$ and $\dist(\Gamma,\cP) > \eps$.

The systematic study of property testing was initiated by Goldreich, Goldwasser
and Ron~\cite{Goldreich:1998:PTC:285055.285060}, and there is a very rich
literature on this topic. For instance, regarding testers, Goldreich and
Trevisan~\cite{Goldreich:2003:TTR:897701.897703} showed that it is sufficient to
consider simpler canonical testers, namely those that randomly choose a subset
$X$ of vertices in $\Gamma$ and then verify whether the induced subgraph
$\Gamma[X]$ satisfies some related property $\mathcal{P'}$. For example, if the
property being tested is having edge density $1/2$, then the algorithm will
choose a random subset $X$ of appropriate size and check whether the edge
density of $\Gamma[X]$ is within, say, $\eps/2$ of $1/2$. Regarding testable
properties, Alon and Shapira~\cite{AS08} proved that every monotone graph
property is testable, and, more generally, that the same holds for hereditary
graph properties~\cite{AS2008}. For more information about property testing, we
refer the reader to~\cite{DBLP:conf/propertytesting/2010}
and the references therein.

In a similar vein, a function $z \colon \cG\to\RR$ from the set $\cG$ of finite
graphs into the real numbers is called a \emph{graph parameter}  if it is
invariant under relabeling of vertices. 
A graph parameter $z\colon \cG\to\RR$ is \emph{estimable} if for every $\eps>0$ and every large enough graph $\Gamma$ with probability at least $2/3$, the value of $z(\Gamma)$ can be approximated up to an additive error of $\eps$ by an algorithm that only has access to a subgraph of $\Gamma$ induced by a set of vertices of size $q_z=q_z(\eps)$, chosen uniformly at random. 
The query complexity of such an algorithm is $\binom{q_z}2$ and the size $q_z$ is called
its \emph{sample complexity}.  Estimable parameters have been considered
in~\cite{FN07} and were defined in the above level of generality in~\cite{Borgs2008}. They are often called \emph{testable parameters}. Borgs
et al.~\cite[Theorem 6.1]{Borgs2008} gave a complete characterization of the
estimable graph parameters which, in particular, also implies that the distance
from monotone graph properties is estimable. Their work uses the concept of graph
limits and does not give explicit bounds on the query complexity required for
this estimation.

We obtain  results for   the bounded graph parameter, which, for a
graph family $\cF$, counts the number of $\cF$-free spanning
subgraphs of the input graph $\Gamma$. Recall that $G'= (V', E')$ is a \emph{spanning subgraph} of a graph $G=(V,E)$ if $V'=V$ and $E' \subseteq E$.

Formally, given a graph $\Gamma \in\cG$ and a family $\cF$ of graphs,
we denote the set of all $\cF$-free spanning subgraphs of $\Gamma$ by~$\Forb(\Gamma,\cF)
= \{\mbox{$G$ is a spanning subgraph of $\Gamma$: } G\in\Forb(\cF) \}$, and we consider the parameter
 \begin{equation}\label{eq:thm2}
 \z(\Gamma)=\frac{1}{|V(\Gamma)|^2}\log_2 |\Forb(\Gamma,\cF)|.
 \end{equation} 
    For example, if $\cF=\{K_3\}$ and $\Gamma=K_n$, computing $\z$ requires
    estimating the number of $K_3$-free subgraphs of $K_n$ up to a multiplicative
    error of $2^{o(n^2)}$: 
    \begin{equation*}
    \z(K_n)=\frac{1}{n^2}\log_2 |\Forb(\Gamma,\cF)|=  \frac{1}{n^2}
    \log_2 2^{\frac12 \binom{n}{2}+o(n^2)} \rightarrow \frac{1}{4}.  
    \end{equation*}
    This was done by
    Erd\H{o}s, Kleitman and Rothschild for $\cF=\{K_k\}$~\cite{EKR1976}, see also
    Erd\H{o}s, Frankl and R\"{o}dl~\cite{EFR1986} for $F$-free subgraphs. 
    Counting problems of this type were studied by several people. Consider for
    instance, the work of Pr\"{o}mel and Steger~\cite{PS92,PS96}, the logarithmic density in Bollob\'{a}s~\cite{Bollobas1998}, 
    and some more recent results about the number of $n$-vertex graphs avoiding  copies of 
    some fixed forbidden graphs~\cite{BBS2004,BS2011}. Algorithmic aspects have been investigated by Duke, Lefmann and R\"{o}dl~\cite{DLR1996} and, quite recently, by Fox, Lov\'{a}sz and Zhao~\cite{FLZ2017}.

    As it turns out, estimating graph parameters $\z(\Gamma)$   is related
    to estimating distances of graphs from the corresponding graph
    property $\cP=\Forb(\cF)$.  Alon, Shapira and Sudakov~\cite[Theorem
    1.2]{ASS09} proved that the distance to every monotone graph property $\cP$ is
    estimable using a natural algorithm, which simply computes the distance
    from the induced sampled graph to~$\cP$. However, one disadvantage of this
    approach is that the accuracy of the estimate relies heavily on stronger versions of 
    Szemer\'{e}di's Regularity Lemma~\cite{Sze76,Alon2000}. Therefore, 
    the query complexity is at least of the order  $\TOWER(\poly(1/\eps))$, by which we mean a
    tower of twos  of height that is polynomial in $1/\eps$. Moreover, it follows 
    from a result of Gowers~\cite{Gowers97} that any approach based on 
    Szemer\'{e}di's Regularity Lemma cannot lead to a bound that is better
    than~$\TOWER(\poly(1/\eps))$.

    In this paper, we introduce the concept of \emph{recoverable} graph properties. Roughly speaking, given a function
    $f\colon (0,1]\to\RR$, we say that a graph property $\cP$ is $f$-recoverable if every
    large graph $G\in\cP$ is $\eps$-close to admitting a partition $\cV$
    of its vertex set into at most $f(\eps)$ classes that witnesses membership in
    $\cP$, i.e., such that any graph that can be partitioned in the same way must be
    in $\cP$. 

    \begin{theorem}\label{thm:2}
        Let $\Forb(\cF)$ be an $f$-recoverable graph property, for some
        function $f\colon (0,1]\to\RR$. Then, for all $\eps>0$ there is $n_0$ such that, for any graph $\Gamma$ with $|V(\Gamma)|\geq n_0$, the graph parameter
        $\z$ defined in~\eqref{eq:thm2}
        can be estimated within an
        additive error of $\eps$ with sample complexity $\poly(f(\eps/6)/\eps)$.
    \end{theorem}

Although one could apply strong versions of regularity to show
that every monotone property $\Forb(\cF)$ is $f$-recoverable, this approach 
would provide an upper bound of at least $\TOWER(\poly(\eps^{-1}))$ for
the function $f$.  We find a connection between this notion of recoverability and the 
graph Removal Lemma, which can lead to better bounds for the function
$f(\eps)$.
The Removal Lemma was first stated explicitly in the literature by Alon
 \etal~\cite{ADLRY1994}  and by F\"{u}redi~\cite{Furedi1995}. The following version, which holds 
 for arbitrary families of graphs was first proven in \cite{AS08}.
\begin{lemma}[Removal Lemma]\label{lemma:removal} 
For every~$\eps > 0$ and every (possibly infinite) family~$\cF$ of graphs, there
exist~$M = M(\eps,\cF)$, $\delta = \delta(\eps,\cF)> 0$ and~$n_0=n_0(\eps,\cF)$
such that the following holds. If a graph~$G$ on~$n \geq n_0$ vertices
satisfies~$\dist(G,\Forb(\cF))\geq \eps$, then there is~$F\in\cF$ with~$|V(F)|\leq
M$ such that $G$ has at least $\delta n^{|V(F)|}$ copies of $F$. $\hfill\qed$
\end{lemma}

Conlon and Fox~\cite{ConlonFox2012} showed that Lemma~\ref{lemma:removal} holds 
with $\delta^{-1},n_0 \leq \TOWER(\poly(\eps^{-1}))$. Although this remains the best 
known bound for the general case, there are many families $\cF$ for 
which Lemma~\ref{lemma:removal} holds with a significantly better
dependency on $\eps$.
For families $\cF=\{F\}$ where $F$ is an arbitrary graph, 
Fox~\cite{Fox2011} (see also~\cite{moshkovitz16:_SRAL}) showed that 
Lemma~\ref{lemma:removal} holds with both~$\delta^{-1}$ and~$n_0$ 
bounded by $\TOWER(O(\log(\eps^{-1})))$ --- as a consequence, this same
bound holds for every \emph{finite} family~$\cF$.
Moreover if $F$ is 
bipartite, than $\delta^{-1}$ and $n_0$ are polynomial in $\eps^{-1}$
and, though it is not possible to get polynomial bounds
when $F$ is not bipartite (see~\cite{Alon2002}), the best known lower bound
for $\delta^{-1}$ is only quasi-polynomial in $\eps^{-1}$. 
Lemma~\ref{lemma:removal} also holds with $\delta^{-1},M,n_0\leq
\poly(\eps^{-1})$ for certain \emph{infinite} families $\cF$. 
For instance, results from~\cite{Goldreich:1998:PTC:285055.285060} 
provide such polynomial bounds when
$\Forb(\cF)$ is the property of ``being $k$-colorable'' (for every
positive integer $k$) or the property of ``having a bisection of size at
most $\rho n^2$'' (for every $\rho>0$) or many other properties that can 
be expressed as ``partition problems''.
    
We show that every monotone graph property $\Forb(\cF)$ is $f$-recoverable for some function $f$ that is only exponential in the
bounds given by the Removal Lemma for the family $\cF$. In fact, we use a weighted version of this lemma (see Lemma~\ref{lemma:rem}).
\begin{theorem}\label{th:recoverable}
For every family~$\cF$ of graphs, the property~$\Forb(\cF)$ is
$f$-recoverable for~$f(\eps) = n_0 2^{\poly(M/\delta)}$,  where $\delta,M$ and
$n_0$ are as in Lemma~\ref{lemma:rem}
with input $\cF$ and $\eps$.
\end{theorem}       
The case of $\cF$ finite is an instance where the bounds given by
Lemma~\ref{lemma:rem} relate polynomially with the bounds of
Lemma~\ref{lemma:removal}. In particular, Theorem~\ref{th:recoverable}, 
together with the abovementioned bounds for Lemma~\ref{lemma:removal} obtained 
by Fox~\cite{Fox2011} for finite families $\cF$, implies that
$\Forb(\cF)$ is $f$-recoverable with $f(\eps) = \TOWER(\poly(\log(1/\eps)))$.

    The remainder of the paper is structured as follows. In Section~\ref{sec_notation} we introduce notation and describe some tools
    that we use in our arguments. In
    Section~\ref{sec_recoverability}, we introduce the concept of recoverable graph properties and prove Theorem~\ref{th:recoverable}.
    Theorem~\ref{thm:2} is a consequence of  Theorem~\ref{th:cardinality}, which is the main result in Section~\ref{sec_distance}. In 
    Section~\ref{sec:proofthm6} we prove Theorem~\ref{th:estimatequotfunc}, which is the technical tool for establishing Theorem~\ref{th:cardinality}. We finish the paper with some concluding remarks in Section~\ref{sec:concluding-remarks}.

    \section{Notation and tools}\label{sec_notation}

    A \emph{weighted graph} $R$ over a (finite) set of
    vertices~$V$ is a symmetric function from~$V\times V$ to~$[0,1]$. A weighted
    graph~$R$ may be viewed as a complete graph (with loops) in which a
    weight~$R(i,j)$ is given to each edge~$(i,j)\in V(R)\times V(R)$, where~$V(R)$
    denotes the vertex set of~$R$. The set of all weighted graphs with vertex
    set~$V$ is denoted by~$\cG^*(V)$ and we define~$\cG^*$ as the union of
    all~$\cG^*(V)$ for~$V$ finite. In particular, a \emph{graph} $G$ is a
    weighted graph such that~$G(i,i)=0$, for every~$i\in V(G)$, and
    either~$G(i,j)=1$ or~$G(i,j)=0$ for every $(i,j)\in V(G)\times V(G)$, $i \neq
    j$. For a weighted graph $R\in\cG^*(V)$ and for sets $A,B \subset V$, we denote
    $e_R(A,B)=\sum_{(i,j) \in A\times B} R(i,j)$ and $e(R) = e(V,V){/2}$.
    Given a graph $G =(V,E)$ and vertex sets $U,W \subseteq V(G)$, let $E_G(U,W)=\{(u,w)\in
     E \colon u\in U, w\in W\}$ and $e_G(U,W)=|E_G(U,W)|$.

        An \emph{equipartition}~$\cV = \{V_i\}_{i=1}^k$ of a weighted graph~$R$ is a
        partition of its vertex set~$V(R)$ such that~$|V_i|\leq |V_j|+1$ for
        all~$(i,j)\in [k] \times [k]$. We sometimes abuse terminology and say that $\cV$
        is a partition of $R$. 
        
        Let~$\cV = \{V_1,\dots,V_k\}$ be an equipartition into $k$ classes  of a graph~$G=(V,E)$.  The
    \emph{cluster graph} of~$G$ by~$\cV$ is a weighted graph~$\quot{G}{\cV} \in
    \cG^*([k])$
    such that $\quot{G}{\cV}(i,j) = e_G(V_i,V_j)/(|V_i||V_j|)$ for all
    $(i,j)\in[k]\times [k]$.
    For a fixed integer~$K>0$, the set of all equipartitions of a vertex
    set~$V$ into at most~$K$ classes will be denoted by~$\Pi_K(V)$. We
    also define the set
    $\quot{G}{\Pi_K} = \{\quot{G}{\cV}: \cV \in \Pi_K(V(G))\}$ of all
    cluster graphs of~$G$ whose vertex set has size at most~$K$. 

    The \emph{distance} between two weighted graphs~$R,R'\in \cG^{*}(V)$ on the same
    vertex set~$V$ is given by 
    \begin{equation*}  \dist(R,R') =
    \frac1{|V|^2}\sum_{(i,j)\in V\times V} |R(i,j)-R'(i,j)|.  \end{equation*}
    For a property $\cH\subseteq \cG^*$  of weighted graphs, i.e., for a subset 
    of the set of weighted graphs which is closed under isomorphisms, we define
    \begin{equation*}  \dist(R,\cH) = \min_{\substack{R'\in
    \cH:\\ V(R')=V(R)}} \dist(R,R').  \end{equation*} 
    Unless said otherwise, we will assume that $\cH$ contains weighted graphs with vertex sets of all possible
    sizes.

    Next, to  set up the version of regularity (or Regularity Lemma) that we use in this work, we use a
    second well-known distance between weighted graphs. Let
    $R_1, R_2\in \cG^*(V)$ be weighted graphs on the same vertex set. The
    \emph{cut-distance} between $R_1$ and $R_2$ is defined as 
    \begin{equation*}
      \displaystyle{\dcut(R_1,R_2)
        = \frac{1}{|V|^2}\max_{S,T\subseteq V}|e_{R_1}(S,T) - e_{R_2}(S,T)|}.
    \end{equation*}

        Let $\Gamma\in\cG(V)$ and $\cV=\{V_i\}_{i=1}^k$ be a partition of $V$. We
        define the weighted graph $\Gamma_{\cV}\in \cG^*(V)$ as the weighted graph
        such that $\Gamma_{\cV}(u,v) = \quot{\Gamma}{\cV}(i,j)$ if $u\in V_i$ and
        $v\in V_j$. Graph regularity lemmas ensure that, for any large graph
        $\Gamma$, there exists an equipartition  $\cV$ into a constant number of classes such
        that $\Gamma_{\cV}$ is a faithful approximation of $\Gamma$. Here, we use
        the regularity introduced by Frieze and Kannan~\cite{FK99}.
        Henceforth we write
    $b = a \pm x$ for $a-x \leq b \leq a+x$.  
        \begin{definition} A partition~$\cV = \{V_i\}_{i=1}^k$ of a graph~$\Gamma$
        is~\emph{$\gamma$-FK-regular} if $\dcut(\Gamma,\Gamma_{\cV}) \leq \gamma$,
        or, equivalently if  for all~$S,T\subseteq
        V(\Gamma)$ it holds that
        \begin{equation*}
        e(S,T) = \sum_{(i,j)\in [k] \times [k]} |S\cap V_i||T\cap
        V_j|\quot{\Gamma}{\cV}(i,j) \pm \gamma |V(\Gamma)|^2.
        \end{equation*}
         \end{definition}
    The concept of $FK$-regularity is also known as \emph{weak regularity}.

    \begin{lemma}[Frieze-Kannan Regularity Lemma]\label{lemma:fk} 
    For every~$\gamma
    > 0$ and every $k_0>0$, there is~$K=k_0 \cdot 2^{\poly(1/\gamma)}$ such that
    every graph $\Gamma$ on $n \geq K$ vertices admits a~$\gamma$-FK-regular
    equipartition into $k$ classes, where $k_0\leq k \leq K$.$\hfill\qed$
    \end{lemma}

    We remark that Conlon and Fox~\cite{conlon12:_bound} found graphs where the
    number of classes in any $\gamma$-FK-regular equipartition is at
    least~$2^{1/(2^{60}\gamma^2)}$ (for an earlier result, see
    Lov\'{a}sz and Szegedy~\cite{LS07}).

    \section{Recoverable parameters}\label{sec_recoverability}

    The main objective of this section is to introduce the concept of
    $\eps$-recoverability and to state our main results in terms of it. 

    \subsection{Estimation over cluster graphs} 
	For a weighted graph $R\in\cG^*(V)$ and a subset $Q \subseteq V$ of vertices, let $R[Q]$ denote the induced weighted subgraph of $R$ with vertex set $Q$. 
	Let us now define estimable parameters in the context of weighted graphs.
    
    \begin{definition} \label{def:g_param} 
        We say that a function~$z\colon\cG^*\to\RR$ (also called a 
        \emph{weighted graph parameter}) is \emph{estimable} with \emph{sample complexity}
        $q\colon (0,1)\to\NN$ if, for every~$\eps>0$ and every weighted
        graph~$\Gamma^* \in\cG^*(V)$ with~$|V|\geq q(\eps)$, we have $z(\Gamma^*) =
        	z(\Gamma^*[Q]) \pm \eps$ with probability at least~$2/3$, where $Q$ is chosen uniformly from all subsets of $V$ of size $q$.
    \end{definition}

    The following result states that graph parameters, that can be expressed as the
      optimal value of some optimization problem over the set $\quot{G}{\Pi_K}$ of all cluster graphs of $G$ of vertex size at most $K$,
      can be estimated with a query complexity that is only exponential 
     in a polynomial in
      $K$ and in the error parameter.

    \begin{theorem}\label{th:estimatequotfunc} Let~$z\colon \cG\to\RR$ be a graph
      parameter and suppose that there is a weighted graph
      parameter~$z^*\colon\cG^*\to\RR$ and constants~$K>0$ and~$c>0$ such
      that
      \begin{enumerate} \item 
        $z(\Gamma) = \max_{R\in\quot{\Gamma}{\Pi_K}} z^*(R)$ for every
        $\Gamma\in\cG$, and \item \label{itm:continuous}
        $|z^{*}(R)-z^{*}(R')| \leq c\cdot\dist(R,R')$ for all weighted
        graphs $R,R'\in\cG^*$ on the same vertex set.
      \end{enumerate}
      Then~$z$ is estimable with sample
      complexity~$\eps\mapsto \poly(K,c/\eps)$.
    \end{theorem} 
    The proof of Theorem~\ref{th:estimatequotfunc} is rather technical and is therefore deferred to Section~\ref{sec:proofthm6}. Moreover, in Section~\ref{sec_distance} we show how to express the parameter $\z$ introduced in~\eqref{eq:thm2}, in terms of the solution of a suitable optimization
    problem over the set $\quot{\Gamma}{\Pi_K}$ of cluster graphs of $\Gamma$ of vertex size at most $K$.

    \subsection{Recovering partitions}

    We are interested in the property of graphs that are free of \emph{copies} of
    members of a (possibly infinite) family~$\cF$ of graphs.  
    To relate this property to a property of cluster graphs, we introduce 
    some definitions. 
    Let~$\phi\colon V(F)\to V(R)$ be a mapping from the set of
    vertices of a graph~$F\in\cG$ to the set of vertices of a weighted
    graph~$R\in\cG^*$. The \emph{homomorphism weight} $\hom_{\phi}(F,R)$ of~$\phi$
    is defined as 
    \begin{equation*} 
        \hom_{\phi}(F,R) = \prod_{(i,j)\in E(F)}R(\phi(i),\phi(j)).  
    \end{equation*} 
    The \emph{homomorphism
    density}~$t(F,R)$ of~$F\in \cG$ in~$R\in \cG^*$ is defined as the average
    homomorphism weight of a mapping in~$\Phi := \{ \phi\colon V(F) \to V(R) \}$,
    that is, 
    \begin{equation*}
        t(F,R) = \frac{1}{|\Phi|}\sum_{\phi\in\Phi} \hom_{\phi}(F,R).  
    \end{equation*} 
    Note that, if $F$ and $R$ are graphs, then~$t(F,R)$ is approximately the
     density of copies of~$F$ in~$R$ (and converges to this quantity when
    the vertex size of~$R$ tends to infinity).
    Since weighted graphs will represent
    cluster graphs associated with a partition of the vertex set of the input graph,
    it will be convenient to work with the following property of weighted graphs:
    \begin{equation*} 
        \Forbhom^*(\cF) = \{R\in\cG^*: t(F,R)=0 \text{ for every $F\in\cF$}\}. 
    \end{equation*}

    Let $R,S\in\cG^*(V)$ be weighted graphs on the same vertex set~$V$.  We say
    that~$S$ is a \emph{spanning subgraph} of~$R$, which will be denoted
    by~$S\leq R$, if~$S(i,j)\leq R(i,j)$ for every~$(i,j)\in V\times V$. 
    When there is no ambiguity, we will just say that $S$ is a
    \emph{subgraph} of $R$.
    We also define
    $\Forbhom^*(R,\cF) = \{S\in \Forb_{\hom}^*(\cF): S\leq R\}$.  

The following result shows that having a cluster graph in~$\Forbhom^*(\cF)$
witnesses membership in~$\Forb(\cF)$. 
\begin{proposition}\label{prop:quothom} 
    Let $\cF$ be a family of graphs and let~$\cV$ be an equipartition of a
    graph~$G$. If $\quot{G}{\cV}\in \Forbhom^*(\cF)$,  then~$G\in \Forb(\cF)$. 
\end{proposition}

\begin{proof} 
    Let $\cV = \{V_i\}_{i=1}^k$ be an equipartition of $G$ and let $R=
    \quot{G}{\cV}$.  Fix an arbitrary element~$F\in\cF$ and an arbitrary
    injective mapping $\varphi\colon V(F)\hookrightarrow V(G)$.  Define the
    function $\psi\colon V(F) \to V(R)$ by $\psi(v)=i$ if $\varphi(v) \in
    V_i$.  Now, if~$t(F,R) = 0$, there must be some edge~$(u,w) \in E(F)$
    such that~$R(\psi(u),\psi(w))=0$, thus  $G(\phi(u),\phi(v))=0$ and
    hence $\hom_{\phi}(F,G) = 0$.  Since~$\phi$ and~$F$ were taken arbitrarily,
    we must have~$G\in\Forb(\cF)$.  
\end{proof}

It is easy to see that the converse of Proposition~\ref{prop:quothom}  does not
hold in general. Indeed, there exist graph families~$\cF$ and graphs $G\in
\Forb(\cF)$ such that~$\quot{G}{\cV}$ is actually very \emph{far} from being
in~$\Forbhom^*(\cF)$ for some equipartition $\cV$ of $G$. As an example,
let $G$ be the $n$-vertex bipartite Tur\'{a}n graph $T_2(n)$ for 
the triangle $K_3$ with partition $V(G)=A \cup B$
and consider $\cV=\{ V_i\}_{i=1}^t $ with $V_i=A_i \cup B_i$, $i=1, \ldots , t$,
where $\{ A_i\}_{i=1}^t$ and $\{ B_i\}_{i=1}^t$ are equipartitions of $A$ and
$B$ respectively. Then $\quot{G}{\cV}$ has  
weight $1/2$ on every edge, so that the distance of
$\quot{G}{\cV}$ to the family $\Forbhom^*(\{K_3\})$ tends to  $1/4$ for $t$ large by
Tur\'{a}n's Theorem. More generally, if $\cV$ is a random equipartition of
a triangle-free graph $G\in\Forb(\{K_3\})$ with large edge density, then with
high probability the cluster graph $\quot{G}{\cV}$ is still approximately $1/4$-far from being in
$\Forbhom^*(\{K_3\})$.

On the other hand, we will prove that there exist partitions for
  graphs in $\Forb(\cF)$ with respect to which an approximate version
  of the converse of Proposition~\ref{prop:quothom} does hold, that
  is, we will prove that every graph in $\Forb(\cF)$ is not too far
  from having a partition into a constant number of classes that witnesses membership in
  $\Forb(\cF)$. We say that such a partition is \emph{recovering} with
  respect to $\Forb(\cF)$.
  {Let us make this more precise.}

\begin{definition} Let $\cP=\Forb(\cF)$ be a monotone graph property. An equipartition~$\cV$ of a 
graph~$G\in\cP$ is
\emph{$\eps$-recovering} for~$\cP$ if
$\dist(\quot{G}{\cV}, \Forbhom^*(\cF)) \leq \eps$.
\end{definition}

\begin{definition}\label{def:recoverable}
  Let $\cP$ be a graph property. For a fixed function~$f\colon (0,1]\to\RR$, we say that the class~$\cP$ is
    \emph{$f$-recoverable} if, for every~$\eps > 0$, there exists
    $n_0=n_0(\eps)$ such that the following holds.
    For every graph~$G\in\cP$ on $n \geq n_0$ vertices, there is an equipartition~$\cV$ of~$G$ into at most~$f(\eps)$ classes which is
    $\eps$-recovering for~$\cP$.
\end{definition}

As a simple example, one can verify that the graph property $\cP
=  \Forb (\cF)$
  of being $r$-colorable is $f$-recoverable for
  $f(\eps) = r/\epsilon$; here and in what follows, for simplicity, we
  ignore divisibility conditions and drop floor and ceiling signs.
  Let~$G$ be a graph in~$\cP$, with color classes $C_1,\dots,C_r$.
  Let~$k=r/\epsilon$.  Start by fixing parts $V_1,\dots,V_t$ of size
  $n/k$ each, with each~$V_i$ contained in some~$C_j$ ($j=j(i)$), and
  leaving out fewer than~$n/k$ vertices from each~$C_j$,
  $1\leq j\leq r$.  The sets~$V_i$, $1\leq i\leq t$, cover a
  subset~$C_j'$ of~$C_j$ and~$X_j=C_j\setminus C_j'$ is left over.  We
  then complete the partition by taking arbitrary parts
  $U_1,\dots,U_{k-t}$ of size $n/k$ each, forming a partition
  of~$\bigcup_{1\leq j\leq r} X_j$.  The cluster graph $\quot{G}{\cV}$
  can be made $r$-partite by giving weight zero to every edge incident
  to vertices corresponding to $U_1,\dots,U_{k-t}$. Therefore
  $\quot{G}{\cV}$ is at distance at most $r/k\leq \eps$ from
  being $r$-partite. Thus, $\dist( \quot{G}{\cV}, \Forb_{hom}^* (\cF)) \leq \eps$ as required.
  
We finish this section by noting that the definition
  of~$f$-recoverable properties has some similarity with the notion of
  \emph{regular-reducible} properties $\cP$ defined by Alon, Fischer,
  Newman and Shapira~\cite{AFNS2009}. When dealing with monotone properties $\cP=\Forb(\cF)$, the main difference is that the
  notion of being regular-reducible requires that every
  graph~$G\in\cP$ should have a \emph{regular} partition such
  that~$\quot{G}{\cV}$ is close to some property~$\cH^*$  of weighted
  graphs, while the definition of $f$-recoverable properties requires only that 
  every $G$ has a partition $\cV$ (regular or not) such that $\quot{G}{\cV}$ is close
  to $\Forbhom^*(\cF)$. Another difference is
  that~$\cH^*$ must be such that having a (regular) cluster graph
  in~$\cH^*$ witnesses only \emph{closeness} to $\cP$, while having a 
  (regular or not) cluster graph in $\Forbhom^*(\cF)$ witnesses membership in $\cP$.

\subsection{Monotone graph properties are recoverable}\label{sec_forbrecoverable}
Szemer\'{e}di's Regularity Lemma~\cite{Sze76} can be used to show that
every monotone (and actually every hereditary) graph property is~$f$-recoverable,
for~$f(\eps) = \TOWER (\poly(1/\eps))$. In the remainder of this section, we prove
that monotone properties~$\cP=\Forb(\cF)$ are recoverable using a weaker version 
of regularity along with the Removal Lemma, which leads to an improvement on the
growth of~$f$ for families~$\cF$ where the Removal Lemma is known to hold with
better bounds than the Regularity Lemma.

We first derive a version of the Removal Lemma stated in the introduction (Lemma~\ref{lemma:removal}) that applies to weighted graphs and homomorphic copies.
\begin{lemma}\label{lemma:rem}
    For every~$\eps > 0$ and every (possibly infinite) family~$\cF$ of 
    graphs, there exist~$\delta=\delta(\eps,\cF)$, $M = M(\eps,\cF)$ and 
    $n_0 = n_0(\eps,\cF)$ such that the following holds.
    If a weighted graph~$R$ such that $|V(R)|>n_0$ satisfies $\dist(R,\Forbhom^*(\cF)) 
    \geq \eps$, then there is a graph~$F\in\cF$ with $|V(F)|\leq M$ such that~$t(F,R)\geq \delta$.
\end{lemma} 

To prove Lemma~\ref{lemma:rem}, we use the following  auxiliary result, which follows from work of Erd\H{o}s and 
Simonovits~\cite{ES1983}.  For completeness we include its proof.

\begin{proposition}\label{prop:blowup}
    Let $\hat{F}$ and $F$ be graphs in $\cG$ such that there is 
    a surjective homomorphism $\zeta \colon V(F)\to V(\hat{F})$. Then, for every 
    graph $H$ such that $t(\hat{F},H)\geq \hat{\delta}$, we must have 
    $t(F,H)\geq \hat{\delta}^{\ell}$, where $\ell = (|V(F)|+1)^{|V(\hat{F})|}$.
\end{proposition}

\begin{proof} We will consider the particular case in which $F$ is obtained from 
   blowing up a single vertex $v$ of $\hat{F}$ into $r$ distinct 
   vertices $v_1,\dots,v_r$ with the same adjacency as $v$, hence
   we assume that $\zeta(v_j)=v$ for every $j=1,\dots,r$ and $\zeta(u)=u$ 
   for every $u\notin \{v_1,\dots,v_r\}$. 

   Let $n=|V(H)|$, $\hat{a} = |V(\hat{F})|$, $a = |V(F)| = \hat{a}+r-1$ and $\hat{F}_{-} =
   \hat{F}-v$ be the graph on $\hat{a}-1$ vertices obtained from $\hat{F}$ by
   deleting $v$. Let $N = t(\hat{F}_-, H)n^{\hat{a}-1}$ be the number of
   homomorphisms from $\hat{F}_-$ to $H$ and $\phi_1,\dots,\phi_N\in
   V(H)^{V(\hat{F}_-)}$ be an enumeration of such homomorphisms. Note that $N\leq
   n^{\hat{a}-1}$ and $N\geq t(\hat{F},H)/n \geq \hat{\delta} n^{\hat{a}-1}$.

   For every $i\in[N]$ and $u\in V(H)$, we consider the function 
   $\phi_i^u$ that extends 
   $\phi_i$ by mapping $v$ to $u$. Define $Z_i=\{u\in
   V(H)\colon \hom_{\phi_i^u}(\hat{F},H)=1\}$ and $z_i = |Z_i|$. We claim 
   there are $z_i^r$ ways of extending $\phi_i$ to a homomorphism 
   from $F$ to $H$. Indeed, every possible extension $\phi'_i\colon V(F)\to V(H)$ 
   of $\phi_i$, such that $\phi'_i(v_j)\in Z_i$, for every $j=1,\dots,r$,
   satisfies $\hom_{\phi'_i}(F,H)=1$. Therefore we have
	   $t(F,H)n^a \geq \sum_{i=1}^N z_i^r$.
   Since $ g(x) =x^r$  is a convex function
   for $x \geq 0$ and $r \geq 1$, we get
   \begin{equation*}
    t(F,H)n^a \geq N\left(\frac{\sum_{i=1}^N z_i}{N}\right)^r.
   \end{equation*}
   Now we use the fact that $\sum_{i=1}^N z_i =
   t(\hat{F},H)n^{\hat{a}}\geq\hat{\delta} n^{\hat{a}}$ and 
   our previous bounds on $N$ to obtain that
   \begin{equation*}
   t(F,H)n^{a}\geq \hat{\delta} n^{\hat{a}-1}\left(\frac{\hat{\delta}
   n^{\hat{a}}}{n^{\hat{a}-1}}\right)^r = \hat{\delta}^{r+1} n^{\hat{a}+r-1} =
   \hat{\delta}^{r+1}
   n^a.
   \end{equation*}
   Therefore, $t(F,H)\geq \hat{\delta}^{r+1}\geq \hat{\delta}^{a+1}$. The general case 
   may be easily obtained by induction on the number of vertices of~$\hat{F}$.
\end{proof}

\begin{proof}[Proof of Lemma~\ref{lemma:rem}]
    Denote by~$\hat{\cF}$ the set of all \emph{homomorphic images} of 
    members of~$\cF$, that is, the set of all graphs~$\hat{F}\in\cG$ such that 
    there is a surjective homomorphism~$F\homto\hat{F}$, for some~$F\in\cF$.
    Let $\hat{M},\hat{\delta}$ and $\hat{n}_0$ be as in Lemma~\ref{lemma:removal} 
    with input $\hat{\cF}$ and $\eps/2$.  We take 
    \begin{equation*}
    M = \max_{\substack{\hat{F}\in\hat{\cF}:\\ |V(\hat{F})|\leq \hat{M}}}
         \min_{\substack{F\in\cF:\\ F\homto \hat{F}}} |V(F)|,
         \end{equation*}
    $n_0 = \hat{n}_0$, $\delta =(\eps/2)^{M^2} \hat{\delta}^{\ell}$, where
    $\ell = (M+1)^{\hat{M}}$.

    Let $R$ be a weighted graph such that $|V(R)| > n_0$ and $\dist(R,\Forbhom^*(\cF)) 
    \geq \eps$. We first define a graph~$H\in\cG(V(R))$ such that $H(i,j) = 1$ if and only
    if~$R(i,j)\geq \eps/2$. It follows from $\dist(R,\Forbhom^*(\cF)) \geq \eps$ that $\dist(H,\Forb(\hat{\cF}))
    \geq \eps/2$. Indeed, suppose to the contrary that there exists 
  $H' \in \Forb(\hat{\cF})$  such that 
  $\dist(H, H') < \eps/2$. Define $R'$ such that $R'(i,j)=R(i,j)$ if $H'(i,j)=1$ and $R'(i,j)=0$ otherwise. By construction, $R'\in \Forbhom^*(\cF)$, and we get a contradiction from
    \begin{align*}
    \dist(R,R')&= \frac{1}{|V(R)|^2} \sum_{\substack{i\in V(R),\\j\in V(R)}} \left|R(i,j)-R'(i,j)\right|\\
    &=  \frac{1}{|V(R)|^2} \left(\sum_{\substack{(i,j):\\H(i,j)=1,\\H'(i,j)=0}}
    \left|R(i,j)-R'(i,j)\right| + \sum_{\substack{(i,j):\\H(i,j)=0,\\H'(i,j)=0}} \left|R(i,j)-R'(i,j)\right| \right)\\
    &\leq \frac{1}{|V(H)|^2} \sum_{\substack{i\in V(H)\\j\in V(H)}} \left|H(i,j)-H'(i,j)\right| + \frac{1}{|V(R)|^2} \sum_{\substack{i\in V(R)\\j\in V(R)}} \frac{\eps}{2} \\
    &= \dist(H,H')+\frac{\eps}{2} <
    \eps.
    \end{align*}
    
    By Lemma~\ref{lemma:removal} there must
    be~$\hat{F}\in\hat{\cF}$, with~$|V(\hat{F})|\leq \hat{M}$, such
    that~$t(\hat{F},H)\geq \hat{\delta}$.  By definition of~$M$, there must
    be~$F\in\cF$ such that~$|V(F)|\leq M$ and there is a surjective homomorphism $F\homto \hat{F}$. It follows 
    from Proposition~\ref{prop:blowup}  that $t(F,H)\geq \hat{\delta}^{\ell}$.  
    Since \[\hom_{\phi}(F,R) \geq (\eps/2)^{|E(F)|}
    \hom_{\phi}(F,H) \geq (\eps/2)^{M^2}   \hom_{\phi}(F,H) \] for
    each~$\phi\colon V(F) \rightarrow {V(R)}$, we must have
    \begin{equation*}
    t(F,R) = \frac{\sum_\phi \hom_\phi(F,R)}{|V(R)|^{|V(F)|}}   
    \geq \left(\frac{\eps}{2}\right)^{M^2} \cdot \frac{\sum_\phi \hom_\phi(F,H)}{|V(H)|^{|V(F)|}} 
    \geq  \left(\frac{\eps}{2}\right)^{M^2}  \cdot \hat{\delta}^{\ell} = \delta.
    \end{equation*}
    \end{proof}

We will use the next result, which states that 
a graph has homomorphism densities close to the ones 
of the cluster graphs with respect to FK-regular partitions.
\begin{lemma}[\protect{\cite[Theorem 2.7(a)]{Borgs2008}}]\label{lemma:counting}
    Let~$\cV$ be a~$\gamma$-FK-regular equipartition of a graph~$G\in\cG$.  Then, for any
    graph~$F\in\cG$ it holds that
    $\displaystyle{t(F,G) =  t(F, G_{\cV}) \pm {4}e(F)\gamma = t(F,\quot{G}{\cV}) \pm {4}e(F)\gamma.}$ 
    $\hfill\qed$
\end{lemma}

We are now ready to prove Theorem~\ref{th:recoverable}, which establishes that every monotone graph property is $f$-recoverable.
\begin{proof}[Proof of Theorem~\ref{th:recoverable}]
    Let $\delta,M$ and $n_0$ be as in Lemma~\ref{lemma:rem} with 
    input $\cF$ and $\eps$ and  let $\gamma = \delta/(3M)^2$. 
    By Lemma~{\ref{lemma:fk}}, it suffices to show 
    that any $\gamma$-FK-regular equipartition $\cV=\{V_i\}_{i=1}^k$
    of a graph $G\in\Forb(\cF)$ into $k\geq n_0$ classes is $\eps$-recovering.

    Let~$R = \quot{G}{\cV}$ and suppose for contradiction that $\dist(R,\Forbhom^*(\cF))\geq \eps$. Then,
    by Lemma~\ref{lemma:rem}, we have $t(F,R)\geq \delta$ for some graph~$F\in\cF$ such
    that~$|F|\leq M$. By Lemma~\ref{lemma:counting}, we  have~$t(F,G)\geq
    \delta-2\gamma M^2 > 0$,  a contradiction to $G \in\Forb (\cF)$.
\end{proof}

\section{Estimation of
\texorpdfstring{$|\Forb(\Gamma,\cF)|$}{cardinality of Forb(Gamma,F)}}
\label{sec_distance}

The objective of this section is to prove Theorem~\ref{thm:2}. To do this, we shall approximate the parameter $\z$ by the solution of an optimization problem as in Theorem~\ref{th:estimatequotfunc}. Recall that
$\Forbhom^*(R,\cF) = \{ S\leq R: t(F,S)=0  \text{ for every } F\in\cF \}$,
and set 
\begin{equation*}
\ex(R,\cF) = \frac1{|V(R)|^2}\max_{S\in \Forbhom^*(R,\cF)} e(S),
\end{equation*}
which measures the largest edge density of a subgraph of $R$ not containing a copy of 
any $F\in\cF$ up to a multiplicative constant.

We shall derive Theorem~\ref{thm:2} from the following auxiliary result.
\begin{theorem}\label{th:cardinality}
    Let $\cF$ be a family of graphs such that $\Forb(\cF)$ is $f$-recoverable
        for some function $f\colon (0,1]\to\RR$. Then, for any $\eps > 0$, there exists 
        $K = f(\poly(\eps))$ and $N = \poly(K)$ such that for any graph $\Gamma$ of vertex size $n \geq N$ it holds that 
        \begin{equation*}
        \frac{\log_2 |\Forb(\Gamma,\cF)|}{n^2} =
        \max_{R\in \quot{\Gamma}{\Pi_K}} \ex(R,\cF) \pm \eps.
        \end{equation*}
    \end{theorem}

    We define the following subsets of edges of a weighted
    graph~$R$:
    \begin{gather}
    E_0(R) = \{(i,j)\in V(R)\times V(R): R(i,j) = 0\} \nonumber\\
    E_1(R) = \{(i,j)\in V(R)\times V(R):R(i,j)>0\}. \nonumber
    \end{gather}

    We will also make use of the \emph{binary entropy function}, defined by
    $H(x) = -x\log_2(x) -(1-x)\log_2(1-x)$ for $0<x<1$. Note that $H(x) \leq -2x \log_2 x$ for $x \leq 1/8$.
    This function has the property (cf. \cite[Corollary 22.2]{Jukna2001}) that  the following inequality holds for $\eps = k/n < 1/2$: 
    \begin{equation}\label{binary-entropy}
    \sum_{i=0}^{k} \binom{n}{i} \leq 2^{H(\eps)n}.
    \end{equation}

    \begin{proof}[Proof of Theorem~\ref{th:cardinality}]
    Let $\cF$ be a family of graphs such that $\Forb(\cF)$ is $f$-recoverable, and
    fix~$\eps > 0$, without loss of generality $ \eps < 1$. Let $\eps'=\eps/18$. Using that  $\log_2x \leq x-1$ for $x<1$ and $H(y) \leq -2y \log_2 y$ for $0 < y \leq 1/8$, we infer $H(\eps') + \eps' \leq -2 \eps' \log_ 2 \eps'
    + \eps' \leq 2 \eps'(1-\eps') +\eps' \leq 3\eps'\leq  \eps/{6}$.     
    We set $K = f(\eps'^2)$ and $N\geq 2K^2/\eps$ big enough so that $\log_2
        N/N < \eps/3$.
        
       Let $\Gamma$  be an $n$-vertex graph, $n \geq N$. We first show that \begin{equation*}\frac{\log_2 |\Forb(\Gamma,\cF)|}{n^2} \geq \max_{R\in
        \quot{\Gamma}{\Pi_K}} \ex(R,\cF) - \eps.\end{equation*}
        Let $R = \quot{\Gamma}{\cV}$ be an arbitrary cluster graph in
        $\quot{\Gamma}{\Pi_K}$ with $\cV = \{V_i\}_{i=1}^k$ for $k\leq K$.
        Choose $S\in\Forbhom^*(R,\cF)$ such that $e(S) = k^2 \ex(R,\cF)$. Further let $G
        \leq \Gamma$ be the subgraph of $\Gamma$ such that $G(r,s) = 0$ if there is a pair $(i,j)\in E_0(S)$ such
        that $r \in V_i$ and $s \in V_j$ and $G(r,s) = \Gamma(r,s)$ otherwise. Thus, we obtain $G$ by deleting all edges from $\Gamma$ between $V_i$ and $V_j$ if $(i,j)\in E_0(S)$.
        
    Since $e(S)$ maximizes $\ex(R,\cF)$ it follows that $\quot{G}{\cV} = S$,
        which implies, by Proposition~\ref{prop:quothom}, that
        $G\in\Forb(\Gamma,\cF)$. 
        Since every subgraph of $G$ also lies in $\Forb(\Gamma,\cF)$, we obtain
        \begin{align*}
            \log_2 |\Forb(\Gamma,\cF)| \geq |e(G)| 
                &= {\frac{1}{2}}\sum_{(i,j)\in[k] \times [k]} S(i,j)|V_i| |V_j| \geq \frac{(n-k)^2}{k^2} e(S) \\
                & \geq \ex(R,\cF)n^2 - kn \geq (\ex(R,\cF) - \eps)n^2.
        \end{align*}
    Note that we used the facts that $e(S)\leq k^2{/2}$ and $n > k/\eps$, as well as $|V_i| \geq n/k-1$ 
    for all $i$.

        Now let us prove the other direction
        \begin{equation*}\frac{\log_2 |\Forb(\Gamma,\cF)|}{n^2} \leq \max_{R\in \quot{\Gamma}{\Pi_K}}
        \ex(R,\cF) + \eps.\end{equation*}
        We first define  $\cU = \bigcup_{G\in\Forb(\Gamma,\cF)} \quot{G}{\Pi_K}$  to
        be the set of all possible cluster graphs of vertex size at most $K$ of graphs in
        $\Forb(\Gamma,\cF)$. 
        Since $\Forb(\cF)$ is $f$-recoverable we can define a function
        \begin{equation*}
        \begin{aligned}
            \eta\colon\Forb(\Gamma,\cF)&\to \Pi_K\times \cU\\
                    G&\mapsto (\cV,T)
        \end{aligned}
        \end{equation*}
        where $\cV$ is an $(\eps'^2)$-recovering partition of $G$ into $k\leq K$
        classes and $T = \quot{G}{\cV}$. Clearly
        \begin{equation}\label{equ:function-count}
            |\Forb(\Gamma,\cF)|\leq |\Pi_K \times \cU|\cdot \max_{(\cV,T)}
            |\eta^{-1}(\cV,T)|.
        \end{equation} 
        Since each mapping from $V(\Gamma)\to [K]$ gives a partition of $V(\Gamma)$ 
        into at most $K$ classes, we have $|\Pi_K|\leq K^n\leq n^n$.
        Moreover, given an arbitrary graph $G\in\cG(V)$ and any partition 
        $\cV$ of $V$, an edge $\quot{G}{\cV}(i,j)$ may assume $n^2$ different values. 
        Hence, we  have 
        $|\cU| \leq n^{2K^2} \leq n^n$.

        Finally we make the following claim, whose proof is deferred for a moment:
        \begin{equation}\label{claim-urbild}
        \log_2 \Big(\max_{(\cV,T)}|\eta^{-1}(\cV,T)| \Big) \leq \(\max_{R\in\quot{\Gamma}{\Pi_K}}\ex(R,\cF) +
        \frac{\eps}{3} \)n^2.
        \end{equation} 
        
        Combining this we can take the logarithm of~\eqref{equ:function-count} to
        get as desired:
        \begin{equation*}
        \begin{aligned}
        \log_2 |\Forb(\Gamma,\cF)| &\leq 
        \log_2 (n^{n}) + \log_2 (n^n) + \(\max_{R\in\quot{\Gamma}{\Pi_K}}\ex(R,\cF) + \frac{\eps}{3}\)n^2 \\
        &\leq   \(\max_{R\in\quot{\Gamma}{\Pi_K}}\ex(R,\cF)  + \eps\)n^2 \hspace{1cm} \mbox{(as $\log_2 n/n \leq \eps/3$).}
        \end{aligned}
        \end{equation*}

        It remains to prove~\eqref{claim-urbild}. 
        To this end, fix $(\cV,T) $ in the image of $\eta$ and let $R =
        \quot{\Gamma}{\cV}$. 
        Choose $S'\in\Forb^{\ast}(R,\cF)$ such that $\dist(T,S')\leq \eps'^2$.
        This is possible because $\cV$ is an $(\eps'^2)$-recovering partition. 
        Set $E_1 = E_1(S')$ and partition~$E_0(S')$ into~$E_0^+:= \{(i,j) \in E_0(S'): T(i,j)> \eps'\}$ and
        $E_0^- = E_0(S') \setminus E_0^+$. Since there 
        are~$b(i,j):=\binom{|V_i||V_j|R(i,j)}{|V_i||V_j|T(i,j)}$  ways to choose $|V_i||V_j|T(i,j)$ edges out of the $|V_i||V_j|R(i,j)$ edges 
        between $V_i$ and $V_j$ in $\Gamma$, we obtain
    \begin{equation}
    |\eta^{-1}(\cV,T)| \leq \prod_{1\leq i<j\leq k}
    b(i,j)  
    \leq \prod_{(i,j)\in E_1} \sqrt{b(i,j)} \prod_{(i,j)\in E_0^+} b(i,j)
    \prod_{(i,j)\in E_0^-} b(i,j).  \label{equ:factors}
    \end{equation}     
        Let us estimate the factors of~\eqref{equ:factors}: 
       
        We can bound each of the factors~$b(i,j)$ of~$E_1$ by  
        $2^{R(i,j)|V_i||V_j|}$.
        Since~$\dist(T,S') \leq \eps'^2$ we have~$|E_0^+| \leq \eps' k^2$, as otherwise it would be the case that
        \begin{equation*}
        \dist(T,S') \geq \sum_{(i,j)\in E_0^+} |T(i,j)-S'(i,j)| > |E_0^+| \eps' \geq  \eps'^2k^2,
        \end{equation*}
        which is a contradiction. Clearly, we have $|E_0^-| \leq k^2$. This
        allows us to upper bound each of the factors of~$E_0^+$ trivially by
        $2^{|V_i||V_j|}$, 
        and each of the factors of~$E_0^-$  by~$2^{H(\eps')
        |V_i||V_j|}$ using~\eqref{binary-entropy}. 
        
    Now let $S\in\Forb^{\ast}(R,\cF)$ be such that 
        \begin{equation*}S(i,j) = \begin{cases} 0 &\text{ if $(i,j)\in E_0(S')$}\\ 
        R(i,j) &\text{ otherwise.}\end{cases}\end{equation*}
            
        Taking the logarithm of~\eqref{equ:factors} and using $|V_i||V_j| \leq
        (n+k)^2/k^2$ we get
        \begin{align*}
            \log_2 |\eta^{-1}(\cV,T)|
                &\leq
                \sum_{(i,j)\in E_1} \frac{R(i,j)}{2}|V_i||V_j| + 
                \sum_{(i,j)\in  E_0^-} H(\eps')|V_i||V_j| 
            + 
                \sum_{(i,j)\in  E_0^+} |V_i||V_j|\\
                &\leq\Big(\kern-2pt
                \sum_{(i,j)\in E_1} \frac{R(i,j)}{2k^2} + 
                \sum_{(i,j)\in  E_0^-} \frac{H(\eps')}{k^2} + \sum_{(i,j)\in  E_0^+}
                \frac{1}{k^2}\Big) (n+k)^2\\
                &\leq  \Big(\frac{1}{2k^2}
                \sum_{(i,j)\in E_1} S(i,j) +H(\eps')+\eps' \Big) (n+k)^2.
        \end{align*}
        Now by using the fact that $S\in\Forb^{\ast}(R,\cF)$ and that 
        $H(\eps')+\eps'\leq \eps/6$ we infer
        \[\log_2 |\eta^{-1}(\cV,T)| \leq \Big(\ex(R,\cF) + \frac{\eps}{6}\Big)(n+k)^2 \leq \Big( \ex(R, \cF) 
                + \frac{\eps}{3}\Big) n^2,\]
 which implies \eqref{claim-urbild}.
         \end{proof}
                
    \begin{proof}[Proof of Theorem~\ref{thm:2}]
    Let $\cF$ be a family of graphs such that $\Forb(\cF)$ is
    $f$-recoverable. Set $K=f(\poly(\eps))$ and $N=\poly(K)$ given
    by Theorem~\ref{th:cardinality} applied to $\eps/3$.
    Theorem~\ref{th:cardinality} ensures that, whenever $\Gamma$ is a graph on $n \geq N$ vertices, we have
    \begin{equation}\label{eq:aux2}
    \left|\frac{\log_2 |\Forb(\Gamma,\cF)|}{n^2} - \max_{R\in \quot{\Gamma}{\Pi_K}} \ex(R,\cF) \right| \leq \frac{\eps}{3}.
    \end{equation}

    Let $\hat{z} \colon \cG \rightarrow \mathbb{R}$ be the graph parameter defined by
    $\hat{z}(\Gamma)=\max_{R\in \quot{\Gamma}{\Pi_K}} z^*(R)$, where
    $z^*(R)= \ex(R,\cF)$. We claim that, given $R$ and $R'$ in $\cG^*(V)$, we have 
    $|z^*(R)-z^*(R')| \leq \dist(R,R')$.
    Indeed, assume without loss of generality that $z^*(R)\geq z^*(R')$ and fix a subgraph $S \leq R$ such that $S \in \Forbhom^*(R,\cF)$ and  $z^*(R)=e(S)/{|V|^2}$. If $S  \in \Forbhom^*(R',\cF)$, we are done, so assume that this is not the case. Let $S'$ be a subgraph of $S$ and $R'$ maximizing $e(S')${, that is, $S'(i,j)=\min\{S(i,j),R'(i,j)\}$}.  Clearly, 
    \begin{equation*}
    e(S')\geq e(S) - {\frac{1}{2}}  \sum_{(i,j) \in V \times V} \left|R(i,j)-R'(i,j)\right| {=} e(S) - {\frac{|V|^2}{2}} \dist(R,R'),
    \end{equation*}
    so that 
    $0 \leq z^*(R)-z^*(R') \leq \left(e(S)-e(S') \right)/|V|^2\leq {\frac{1}{2}} \dist(R,R').$

    This allows us to apply Theorem~\ref{th:estimatequotfunc} to conclude that  $\hat{z}$ is
    estimable with sample complexity $q(\eps)=\poly(K,1/\eps)$.
    Let $Q$ be  chosen uniformly from all subsets of $V$ of size 
    $q'=\max\{q(\eps/3),N\}$ and set $\bar{\Gamma} = \Gamma[Q]$.
    It follows that, with probability at least $2/3$, we have
     $|\hat{z}(\bar{\Gamma}) -
    \hat{z}(\Gamma)| \leq \eps/3$.   By~\eqref{eq:aux2} we have
    $\left|n^{-2}\log_2|\Forb(\Gamma,\cF)|-\hat{z}(\Gamma)\right|\leq \eps/3$. On the other hand, we
    can also apply~\eqref{eq:aux2} to $\bar{\Gamma}$ to obtain
    $\left|\hat{z}(\bar{\Gamma}) - {q'}^{-2}\log_2|\Forb(\bar{\Gamma},\cF)|\right|\leq \eps/3$.  By adding the
    last three inequalities, we get that \begin{equation*}\left|\frac1{n^2}\log_2|\Forb(\Gamma,\cF)|-
    \frac{1}{{q'}^{2}}\log_2|\Forb(\bar{\Gamma},\cF)|\right| \leq \eps,\end{equation*}
    as required.
    \end{proof}

\section{Proof of Theorem~\ref{th:estimatequotfunc}}\label{sec:proofthm6} 

    Here we will prove Theorem~\ref{th:estimatequotfunc}. Its  proof  is based
    on the following lemma, which asserts that the set of cluster graphs of a
    graph~$\Gamma$ is very `similar' to the set of cluster graphs of `large
    enough' samples of~$\Gamma$.
    
    \begin{lemma}\label{lemma:sample}
      Given~$K > 0$,~$\eps >0$ there is $q = \poly(K,1/\eps)$ such that the
      following holds. Consider a graph~$\Gamma$ whose vertex set $V$ has cardinality $n \geq q$  and
      a random subgraph $\bar{\Gamma} = \Gamma[\bar{V}]$, where $\bar{V}$
        is chosen uniformly from all subsets of $V$ of size $q$.
      Then, with probability at least~$2/3$, we have
      \begin{enumerate}
      \item for each~$\cV \in \Pi_K(V)$, there is a~$\bar{\cV} \in
        \Pi_K(\bar{V})$ with~$\dist(\quot{\Gamma}{\cV},\quot{\bar{\Gamma}}{\bar{\cV}})
        \leq \eps$, and
        
      \item for each~$\bar{\cV} \in \Pi_K(\bar{V})$, there is a~$\cV \in
        \Pi_K(V)$  with~$\dist(\quot{\Gamma}{\cV},\quot{\bar{\Gamma}}{\bar{\cV}}) \leq
        \eps.$ 
      \end{enumerate}
    \end{lemma}

    \def\AlmostReducible#1#2{\cG_{#1}^{(#2)}}
    \def\cT{\mathcal{T}}
    \def\PartitionsOf#1#2{\Pi_{=#2}(#1)}

    For a set of vertices $V$ and an integer $k$, define $\PartitionsOf{V}{k}$ 
    as the set of all equipartitions of $V$ of size \emph{exactly} $k$.
    For every $R\in\cG^*([k])$ and $\eps\geq 0$, we define the property
    \[\AlmostReducible{R}{\eps} = \left\{G\in\cG: \text{$\exists\cV\in\PartitionsOf{V(G)}{k}$ 
    such that $\dist(\quot{G}{\cV},R)\leq \eps$}\right\}\]
    of all graphs admitting a reduced graph which is $\eps$-close to $R$. 
    Note that $\AlmostReducible{R}{\eps}$ contains no graphs of size less than $k$. In particular, if $|V(G)| < k$, then $\dist(G,\AlmostReducible{R}{\eps})=\infty$. Since we will compare $\AlmostReducible{R}{\eps}$ only with large graphs, this is not a problem.

	The following theorem is a consequence of a more general result of~\protect{\cite[Theorem 2.7]{FMS2010}}.
	For our application it suffices to state this result in the case of simple graphs ($r=2$, $s=1$) with \emph{density tensor} $\Psi = \{S\in \cG^*([k]): \dist(R,S)\leq \eps\}$. 
    \begin{theorem}[\protect{\cite[Theorem 2.7]{FMS2010}}]\label{thm:FMS}
        For every positive integer $k$, and every $\eps>0$ and $\delta>0$, there is 
        $q'=q'(k,\eps,\delta)=\log^3(\delta^{-1})\cdot\poly(k,\eps^{-1})$
        such that the following holds.
        For every $R\in\cG^*([k])$ there is a randomized algorithm
        $\cT$ which takes as input an oracle access to a graph $G$ of size at least $k$ and
        satisfies the following properties:
        \begin{enumerate}
            \item If $G\in\AlmostReducible{R}{\eps}$, then $\cT$
            \textsc{accepts} $G$ with probability at least $1-\delta$.
            \item If $\dist(G,\AlmostReducible{R}{\eps})>\eps$, then $\cT$
            \textsc{rejects} $G$ with probability at least
            $1-\delta$.
        \end{enumerate}
        The query complexity of $\cT$ is bounded by $q'$.\hfill\qed
    \end{theorem}
    
    \def\qfms{q_{\ref{cor:fms}}}
    \begin{corollary}\label{cor:fms}
        For every positive integer $k$, and any $\eps>0$ and $\delta>0$,
        there is an integer~$q = \qfms(k, \eps,\delta)=\poly(k,1/\eps,\log^3(1/\delta))$
        such that for every~$R\in\cG^*([k])$ and every graph $G\in\cG(V)$, 
        with~$|V|\geq q$, we have 
        \begin{enumerate}
            \item If $G\in \AlmostReducible{R}{\eps}$, then 
            $\PP(\dist(G[Q], \AlmostReducible{R}{\eps}) > \eps) < \delta$.
            \item If $\dist(G,\AlmostReducible{R}{\eps})>\eps$, then 
            $\PP(G[Q]\in \AlmostReducible{R}{\eps}) < \delta$.
        \end{enumerate}
    \end{corollary}
    \begin{proof}
        Fix $R\in\cG^*([k])$ and let $\cT$ be a tester for the property 
        $\AlmostReducible{R}{\eps}$ as in the statement of Theorem~\ref{thm:FMS}, with query complexity $q'(k,\eps,\delta)$. 

        It follows from a result of Goldreich and Trevisan~\cite[Theorem 2]{Goldreich:2003:TTR:897701.897703} (see also~\cite{GT2003errata}), 
        that there is \emph{canonical tester} $\cT'$ for $\AlmostReducible{R}{\eps}$ with 
        sample complexity $q(k,\eps,\delta)=\poly(q'(k,\eps,\delta))$, i.e., 
        a tester that simply chooses a set $Q \in \binom{V}{q}$ uniformly at random and then accepts the input if and only if 
        $G[Q]$ satisfies a certain property \textsc{Acc} of graphs of size~$q$. 

        To prove (1), if $G\in\AlmostReducible{R}{\eps}$ then we get $\PP(G[Q]\notin
        \textsc{Acc})<\delta$. Moreover, if $Q$ is a set of size~$q$ such
        that $\dist(G[Q],\AlmostReducible{R}{\eps})>\eps$, then 
        $G[Q]\notin \textsc{Acc}$ --- because $G[Q]$ must be rejected (with 
        probability~$1$) when given as \emph{input} to $\cT$. So $\PP(\dist(G[Q], \AlmostReducible{R}{\eps}) > \eps) < \delta$.

        Analogously, if $\dist(G,\AlmostReducible{R}{\eps})>\eps$, 
        then $\PP(G[Q]\in\textsc{Acc})<\delta$. Moreover, 
        if $Q$ is a set of size $q$ such that 
        $G[Q]\in\AlmostReducible{R}{\eps}$, then 
        $G[Q]\in \textsc{Acc}$ --- because $G[Q]$ must be accepted (with 
        probability~$1$) when given as \emph{input} to $\cT$. So $\PP(G[Q]\in \AlmostReducible{R}{\eps}) < \delta$.
    \end{proof}

\begin{lemma}\label{prop:distquots}
   For $n > 2k$, let~$G_1,G_2\in\cG(V)$ with $|V| = n$ and let~$\cV \in \PartitionsOf{V}{k}$. Then, it is~$\dist(\quot{G_1}{\cV}, 
    \quot{G_2}{\cV}) \leq \dist(G_1,G_2) + 2k/(n-2k)$.
$\hfill\qed$
\end{lemma}
\begin{proof}
    \begin{align*}
        \dist(\quot{G_1}{\cV}, \quot{G_2}{\cV}) 
            &= \frac1{k^2}\sum_{(i,j)\in [k]^2} \left|\quot{G_1}{\cV}(i,j)-\quot{G_2}{\cV}(i,j)\right|\\
            &\leq \frac1{k^2}\sum_{(i,j)\in[k]^2} \frac{|e_{G_1}(V_i,V_j) - e_{G_2}(V_i,V_j)|}{\frac{(n-k)^2}{k^2}}\\
            &\leq \frac1{(n-k)^2} \sum_{(i,j)\in [k]^2}\sum_{\substack{u\in V_i\\v\in V_j}} |G_1(u,v)-G_2(u,v)| \\
            &\leq \(1+\frac{2k}{n-2k}\)\frac1{n^2} \sum_{(u,v)\in V^2} |G_1(u,v)-G_2(u,v)|\\
            &= \(1+\frac{2k}{n-2k}\)\dist(G_1,G_2).
    \end{align*}
\end{proof}
    
    \begin{proof}[Proof of Lemma~\ref{lemma:sample}]
    \def\cB{\mathcal{B}}
    Fix $\eps>0$ and $K$ as in the statement of the lemma. Let $\delta=\frac{1}{6K} \cdot(\eps/4)^{K^2}$ and take 
    \[q = \max\left\{\frac{8K}{\eps}+2K, \qfms\left(K,\frac14\eps,\delta\right)\right\} = \poly(K,\eps).\]

    Let $1\leq k\leq K$. Fix a family $\cR\subseteq \cG^*([k])$ such that, for every $S\in\cG^*([k])$, there is
    $R\in\cR$ such that $\dist(R,S)\leq \eps/4$. There is one such family with
    cardinality at most $(4/\eps)^{k^2}$. 
    
    Let $\Gamma$ be a graph with vertex set $V$, where $|V|\geq q$. Let $Q\in\binom{V}{q}$ be chosen uniformly at random and consider the following events
	\begin{enumerate}
		\item $E' = \left[\text{$\exists \cV\in\PartitionsOf{V}{k}$ satisfying $\dist(\quot{\Gamma}{\cV},\quot{\Gamma[Q]}{\cV'})>\eps$ 
			for \emph{every} $\cV'\in \PartitionsOf{Q}{k}$}\right]$,
		\item $E = \left[\text{$\exists\cV'\in\PartitionsOf{Q}{k}$ satisfying $\dist(\quot{\Gamma}{\cV},\quot{\Gamma[Q]}{\cV'})>\eps$ 
			for \emph{every} $\cV\in \PartitionsOf{V}{k}$}\right]$.
	\end{enumerate}
	We claim that these two events occur each with probability less than $1/(6K)$.
	It then follows by taking the union bound over $k=1,\dots K,$ that $Q$ satisfies both (1) and (2) of the statement of Lemma~\ref{lemma:sample} with probability at least $1-1/6-1/6 = 2/3$.

	We only prove that $\PP(E)\leq 1/(6K)$.
    An analogous argument shows that $\PP(E')\leq 1/(6K)$.
    Suppose that event $E$ happens and let $\cV'\in\PartitionsOf{Q}{k}$ be as in (2).
    Define $S=\quot{\Gamma[Q]}{\cV'}$ and let $R\in\cR$ be such that 
    $\dist(R,S)\leq \eps/4$. Since 
    $\Gamma[Q]\in\AlmostReducible{S}{0}$ and $\Gamma \notin \AlmostReducible{S}{\eps}$, 
    the triangle inequality implies that
    $\Gamma[Q]\in\AlmostReducible{R}{\eps/4}$ and 
    $\Gamma\notin \AlmostReducible{R}{3\eps/4}$. 
    Therefore    
    \begin{align*}
    \PP(E) &\leq \PP\(\text{$\exists R\in\cR: \Gamma[Q]\in\AlmostReducible{R}{\eps/4}$ and  $\Gamma\notin \AlmostReducible{R}{3\eps/4}$}\).
    \end{align*}
    We claim that if $\Gamma\notin \AlmostReducible{R}{3\eps/4}$, then $\dist(\Gamma,\AlmostReducible{R}{\eps/4})>\eps/4$.
    To show this, consider the contrapositive statement and let $\Gamma'\in \AlmostReducible{R}{\eps/4}$ such that $\dist(\Gamma,\Gamma')\leq \eps/4$. 
    By definition there is an equipartition $\cV' \in \PartitionsOf{V}{k} $ such that $\dist(\quot{\Gamma'}{\cV'}, R)\leq \eps/4$. 
    In addition Lemma~\ref{prop:distquots} implies that \[\dist(\quot{\Gamma}{\cV'}, \quot{\Gamma'}{\cV'})\leq \dist(\Gamma,\Gamma')+ 2k/(|V|-2k)\leq \eps/4+2K/(q-2K) \leq \eps/2.\]
    It follows from the triangle inequality that $\dist(\quot{\Gamma}{\cV'},R)\leq 3\eps/4$.
    Therefore
    \begin{align*}
     \PP(E) &\leq \PP\(\exists R\in\cR: \Gamma[Q]\in\AlmostReducible{R}{\eps/4} 
		    \text{ and } \dist(\Gamma,\AlmostReducible{R}{\eps/4})>\eps/4\) \\
		    &\leq \sum_{R \in \cR} \PP \( \Gamma[Q]\in\AlmostReducible{R}{\eps/4} 
		    \text{ and } \dist(\Gamma,\AlmostReducible{R}{\eps/4})>\eps/4\)\\
		    &\leq \delta|\cR| \leq 1/(6K) 
    \end{align*}
    where the last line comes from Corollary~\ref{cor:fms}(2) with $\dist(\Gamma,\AlmostReducible{R}{\eps/4}) > \eps/4$.
    
    \end{proof}

We now deduce Theorem~\ref{th:estimatequotfunc} from Lemma~\ref{lemma:sample}.

    \begin{proof}[Proof of Theorem~\ref{th:estimatequotfunc}]
        Fix~$\eps>0$ and an input graph~$\Gamma\in\cG(V)$. Let~$q$ be as in 
        Lemma~\ref{lemma:sample} with input~$K$ and~$\eps/c$.
        Choose $Q$ uniformly from all subsets of $V$ of size $q$ and set $\bar{\Gamma} = \Gamma[Q]$.
        We will show that  
        $z(\Gamma) = z(\bar{\Gamma}) \pm \eps$ with probability at least $2/3$. 
        
        Let~$\cV\in\Pi_K(V)$ be an equipartition of~$\Gamma$ such
        that~$z(\Gamma) = z^*(\quot{\Gamma}{\cV})$. By
        Lemma~\ref{lemma:sample}(1), with probability at least~$2/3$, there is
        a partition~$\bar{\cV}$ of~$\bar{\Gamma}$ such that
        $\dist(\quot{\Gamma}{\cV},\quot{\bar{\Gamma}}{\bar{\cV}}) <
        \eps/c$.
        By the second condition on $z^*$ in the statement of
        Theorem~\ref{th:estimatequotfunc}, we have
        $|z^{\ast}(\quot{\bar{\Gamma}}{\bar{\cV}}) - z^{\ast}
        (\quot{\Gamma}{\cV})| \leq \eps$,
        and therefore
        $z(\bar{\Gamma}) \leq z^{\ast}(\quot{\bar{\Gamma}}{\bar{\cV}})
        \leq z^{\ast} (\quot{\Gamma}{\cV}) + \eps = z(\Gamma)+\eps$.

        A symmetric argument {relying on Lemma~\ref{lemma:sample}(2)} shows that~$z(\Gamma) \leq z(\bar{\Gamma})+\eps$.
    \end{proof}

    \section{Concluding remarks}
    \label{sec:concluding-remarks}

    In this paper, we introduced the concept of $f$-recoverability of a graph
    property $\cP$. Using this concept, and the fact that any monotone property
    $\cP=\Forb(\cF)$ is recoverable for a function $f$ whose size is given by
    the Graph Removal Lemma, we found a probabilistic algorithm to estimate the
    number of $\cF$-free subgraphs of a large graph $G$ whose sample complexity does not depend on regularity.

    Being a new concept, little is known about $f$-recoverability itself, and we
    believe that it would be interesting to investigate this notion in more detail.
    For instance, in our proof that any monotone property $\Forb(\cF)$ is
    $f$-recoverable, we found $\eps$-recovering partitions $\cV$ that were
    $\gamma$-FK-regular (in fact, we showed that \emph{any} such partition is
    $\eps$-recovering), where~$\gamma(\eps)$ is chosen in such a way that the
    Removal Lemma applies. On the other hand, our discussion after Definition 3.5
    implies that the property of being $r$-colorable is $\eps$-recoverable with
    sample complexity $r/\eps$, and thus we may find an $\eps$-recovering
    partition whose size is less than the size required to ensure the existence of
    an FK-regular partition. It is natural to ask for properties that can be
    recovered by small partitions; more precisely, one could ask for a characterization
    of properties that are $f(\epsilon)$-recoverable for~$f(\epsilon)$ polynomial
    in~$1/\epsilon$.

    Here, we  restricted ourselves to monotone graph properties. We should
    mention that the parameter 
    \[\hat{z_{\cP}}(\Gamma) = \frac1{|V(\Gamma)|^2}\log_2|\{G\leq \Gamma: G\in \cP\}|\]
    might not even be estimable for arbitrary (non-monotone) properties $\cP$. 
    For instance, if $\cP$ is the hereditary property of graphs 
    having no independent sets of size three, then
    the complete graph~$K_n$ and the graph~$K_n-E(K_3)$, which is obtained from $K_n$ by removing the edges of a triangle, have
    quite a different number of spanning subgraphs satisfying~$\cP$,
    namely~$2^{n^2/4}$ and~$0$, respectively, although their edit distance is
    negligible.  It follows from \cite[Theorem 6.1]{Borgs2008} that~$\hat{z_{\cP}}$ is not estimable. 

    Nevertheless, the definition of $f$-recoverable can be extended to cope 
    with general hereditary properties, which, along with
    Theorem~\ref{th:estimatequotfunc}, provides a way of estimating other
    interesting hereditary properties.  In particular, this framework is used
    in a follow-up paper to estimate the edit distance to any fixed hereditary
    property with a sample complexity similar to the one obtained here. We should mention here
    that, given a monotone property $\Forb(\cF)$, the parameter $z_{\cF}$ is actually closely 
    related to the parameter $d_{\cF}\colon\Gamma\mapsto\dist(\Gamma,\Forb(\cF))$. In fact, 
    $\eps$-recovering partitions along with techniques analogous to the ones 
    used in \cite{BT1997} can be used to show that, for any  graph $\Gamma = (V, E)$, we have
    \[d_{\cF}(\Gamma) = \frac{2|E|}{|V|^2} -2z_{\cF}(\Gamma)\pm o(1),\]
      which implies that  estimating $z_{\cF}$ provides an indirect way 
    for estimating $d_{\cF}$.
    
    \bibliographystyle{amsplain}
\providecommand{\bysame}{\leavevmode\hbox to3em{\hrulefill}\thinspace}
\providecommand{\MR}{\relax\ifhmode\unskip\space\fi MR }
\providecommand{\MRhref}[2]{%
  \href{http://www.ams.org/mathscinet-getitem?mr=#1}{#2}
}
\providecommand{\href}[2]{#2}

    \end{document}